\documentclass{amsart}

\usepackage{amssymb,amsmath,amsthm,amsfonts}
\usepackage{graphicx}
\usepackage{amsthm}
\usepackage{amsmath}
\usepackage{color}
\usepackage{mathabx}
\usepackage[all,cmtip]{xy}
\usepackage{tikz-cd}
\usepackage{stmaryrd}
\usepackage{hyperref}
\usepackage{algorithm}
\usepackage{algpseudocode}
\usepackage{pifont}
\usepackage[margin=1in]{geometry}
\usepackage{mathtools}
\usepackage[percent]{overpic}

\newtheorem{thm}{Theorem}[section]

\newtheorem{lem}[thm]{Lemma}
\newtheorem{cor}[thm]{Corollary}

\newtheorem{remark}[thm]{Remark}
\newtheorem{question}[thm]{Question}

\newcommand{\R}{\mathbb{R}}
\newcommand{\Z}{\mathbb{Z}}

\newcommand{\mathh}{\mathbb{H}}
\newcommand{\C}{\mathbb{C}}

\newcommand{\Gr}{\mathrm{Gr}_2(\mathcal{V})}
\newcommand{\St}{\mathrm{St}_2(\mathcal{V})}

\newcommand{\qi}{\textbf{i}}
\newcommand{\qj}{\textbf{j}}
\newcommand{\qk}{\textbf{k}}

\begin{document}

\markboth{Tom Needham}
{Knot Types of Generalized Kirchhoff Rods}

%%%%%%%%%%%%%%%%%%%%% Publisher's Area please ignore %%%%%%%%%%%%%%
%\catchline{}{}{}{}{}
%%%%%%%%%%%%%%%%%%%%%%%%%%%%%%%%%%%%%%%%%%%%%%%%%%%%%%%%%%%%%%%%%%%

\title{Knot Types of Generalized Kirchhoff Rods}

\author{Tom Needham}

\address{Department of Mathematics, The Ohio State University \\
100 Math Tower, 231 West 18th Avenue, Columbus OH, 43210-1174 \\
needham.71@osu.edu}

\maketitle

\begin{abstract}
Kirchhoff energy is a classical functional on the space of arclength-parameterized framed curves whose critical points approximate configurations of springy elastic rods. We introduce a generalized functional on  the space of framed curves of arbitrary parameterization, which model rods with axial stretch or cross-sectional inflation. Our main result gives explicit parameterizations for all periodic critical framed curves for this generalized functional. The main technical tool is a correspondence between the moduli space of shape similarity classes of closed framed curves and an infinite-dimensional Grassmann manifold. The critical framed curves have surprisingly simple parameterizations, but they still exhibit interesting topological features. In particular, we show that for each critical energy level there is a one-parameter family of framed curves whose base curves pass through exactly two torus knot types, echoing a similar result of Ivey and Singer for classical Kirchhoff energy. In contrast to the classical theory, the generalized functional has knotted critical points which are not torus knots. 
\end{abstract}

%%%%%%%%%%%%%%%%%%%%%%%%%%%%%%%%%%%%%%%%%%%%%%%%%%%%%%%%%%%%%%%%%%%
\section{Introduction}
%%%%%%%%%%%%%%%%%%%%%%%%%%%%%%%%%%%%%%%%%%%%%%%%%%%%%%%%%%%%%%%%%%%

With roots dating back to Euler and the Bernoullis, the \emph{Kirchhoff elastic rod} problem is one of the oldest and most widely-studied in the calculus of variations. A Kirchhoff rod is a configuration of a bendable, twistable rod which is energy-minimizing amongst configurations with prescribed boundary conditions. Kirchhoff rods are modelled mathematically as framed space curves which are critical with respect to the action functional given by the sum of total squared curvature and total squared twist (defined below). The Kirchhoff rod problem and related problems in elasticity have applications to biology \cite{benham}, computer graphics \cite{bergou}, high energy physics \cite{harland}, robotics \cite{jones} and fluid mechanics \cite{kida} and are interesting from a purely mathematical perspective \cite{bryant,ivey,langer}; for example, it is shown in \cite{ivey} that the knot types realized by closed Kirchhoff rods are exactly the torus knots.

In this paper we introduce a generalized energy functional on the space of \emph{parameterized} framed curves, which can be viewed as models for elastic rods allowing axial stretch or cross-sectional inflation. A common trick in elastic rod theory is to represent framed paths in $\R^3$ as paths in the quaternions (e.g., \cite{dichmann,hanson,hu}) and we show that the periodic critical points of the generalized energy functional take a surprisingly simple form in this representation. The change of coordinates has the remarkable property that it induces an identification of the space of periodic framed curves with an infinite-dimensional Grassmann manifold. We use this identification to solve the variational problem by reformulating it as a basic ODE with explicit solutions given by finite Fourier curves. Despite their simple parameterizations, the critical framed curves for the generalized functional exhibit complex topological behavior. We show that every torus knot type is realized as a generalized elastic rod, mirroring the classical result of \cite{ivey}. On the other hand, generalized rods are shown to exhibit knot types which do not appear in the classical theory. We give a precise formulation of the Kirchhoff rod problem and of our generalized version below.

%%%%%%%%%%%%%%%%%%%%%%%%%%%%%%%%%%%%%%%%%%%%%%%%%%%%%%%%%%%%%%%%%%%
\subsection{Classical Kirchhoff Energy}
%%%%%%%%%%%%%%%%%%%%%%%%%%%%%%%%%%%%%%%%%%%%%%%%%%%%%%%%%%%%%%%%%%%

Let $(\gamma,V)$ be a framed curve in $\R^3$---that is, $\gamma:I \rightarrow \R^3$ ($I$ an arbitrary interval) is a smooth, parameterized, immersed curve and $V:I \rightarrow \R^3$ is a unit normal vector field along $\gamma$. We denote the parameter by $t \in I$. The \emph{Kirchhoff elastic energy} of $(\gamma,V)$ is given by
$$
\mathrm{E}_{Kir}(\gamma,V)=\int_I \kappa^2 + \mathrm{tw}^2 \; \mathrm{d}s.
$$
We use $\mathrm{d}s=\|\gamma'(t)\| \mathrm{d}t$ to denote \emph{measure with respect to arclength}, where $\|\cdot\|$ denotes the Euclidean norm. The terms $\kappa$ and $\mathrm{tw}$ denote the \emph{curvature} and \emph{twist rate} of $(\gamma,V)$, respectively. These are given by the formulas
\begin{equation}\label{eqn:curvature_and_twist}
\kappa=\left\| D_s^2 \gamma \right\| \;\;\; \mbox{ and } \;\;\; \mathrm{tw}=\left<D_s V, D_s \gamma \times V \right>.
\end{equation}
The operator $D_s = \frac{1}{\|\gamma'\|} \frac{d}{dt}$ denotes \emph{derivative with respect to arclength} and $\left<\cdot,\cdot\right>$ and $\cdot \times \cdot$ denote the standard inner product and cross-product on $\R^3$, respectively. 

The Kirchhoff energy functional restricts to the subspace of framed curves $(\gamma,V)$ with $\gamma$ arclength-parameterized, and critical points of the restriction (with prescribed boundary conditions) are  referred to as \emph{(inextensible) Kirchhoff elastic rods}. These framed curves approximate the strain-minimizing  configuration of a bendable, twistable, naturally straight rod with the given boundary conditions. The curve $\gamma$ is referred to as the \emph{centerline} of the elastic rod $(\gamma,V)$.

Rather than attempting to survey the huge amount of literature on elastic rod theory (see \cite{antman, maddocks} as starting points), we mention a particular result which served as a main inspiration for this paper. In \cite{ivey}, Ivey and Singer prove:

\begin{thm}\label{thm:ivey_singer}
For every pair of relatively prime integers $(h,k)$ there is a regular homotopy of periodic Kirchhoff elastic rods such that the the centerlines of the rods pass from an $h$-times-covered circle to a $k$-times-covered circle. Moreover, the centerlines in each family pass through exactly two torus knot types, every torus knot type is realized in this manner, and these are the only knot types realized by elastic rod centerlines.
\end{thm}

%%%%%%%%%%%%%%%%%%%%%%%%%%%%%%%%%%%%%%%%%%%%%%%%%%%%%%%%%%%%%%%%%%%
\subsection{Generalized Kirchhoff Energy and Inflatable Rods}\label{sec:generalized_energy_functional}
%%%%%%%%%%%%%%%%%%%%%%%%%%%%%%%%%%%%%%%%%%%%%%%%%%%%%%%%%%%%%%%%%%%

It is clear from its definition that $\mathrm{E}_{Kir}$ is invariant under reparameterizations (that is, under the action of the diffeomorphism group of $I$ by precomposition of $\gamma$ and $V$ with a common diffeomorphism). In this paper we study an energy functional $\mathrm{E}$ which distinguishes elements of the same diffeomorphsim group orbit. It is defined for a parameterized framed curve $(\gamma,V)$ by
$$
\mathrm{E}(\gamma,V)=\int_I \left(\kappa^2 + \mathrm{tw}^2 + \mathrm{st}^2\right) \|\gamma'\|^2 \mathrm{d}s,
$$
where $\mathrm{st}$ denotes the \emph{relative stretch rate} of $(\gamma,V)$, defined by the formula
\begin{equation}\label{eqn:stretch_rate}
\mathrm{st} = \frac{1}{\|\gamma'\|} D_s \|\gamma'\|.
\end{equation}
The restriction of $\mathrm{E}$ to the subset of arclength-parameterized framed curves reduces to $\mathrm{E}_{Kir}$, so we refer to $\mathrm{E}$ as \emph{generalized Kirchhoff energy}.

Framed curves with arbitrary parameterizations are typically studied as models of  \emph{extensible elastic rods}---that is, rods which allow for axial stretch---but $\mathrm{E}$ is not the usual energy functional in this setting. A traditional energy functional on the configuration space of extensible elastic rods is of the form
$$
\mathrm{E}_{ext}(\gamma,V) = \int_I \kappa^2 + \mathrm{tw}^2 + (\|\gamma'\|-1)^2 \; \mathrm{d}s,
$$
where the assumption is that the relaxed rod is arclength-parameterized and the modified energy functional accordingly includes a Hooke's Law term \cite{antman}. 

The physical meaning of the energy functional $\mathrm{E}$ is more clear if we interpret the space of parameterized framed curves as the configuration space of \emph{inflatable elastic rods}; that is, elastic rods which allow cross-sectional radius inflation. These are a simple case of rods with deformable cross-section, which have been of recent interest for their applications to continuum soft robotics \cite{burgner} and modelling of carbon nanotubes \cite{kumar}. 

An inflatable elastic rod is represented as a triple $(\gamma,V,r)$ with $(\gamma,V)$ an arclength-parameterized framed curve and $r:I \rightarrow \R_{>0}$ a \emph{cross-sectional radius function}.  The set of parameterized framed curves of fixed length $\mathrm{length}(I)$ maps injectively into the set of inflatable framed curves via $(\gamma,V) \mapsto \left(\widetilde{\gamma},\widetilde{V},\|\gamma'\|\right)$, where $(\widetilde{\gamma},\widetilde{V})$ is the unique arclength-reparameterization of $(\gamma,V)$. The image of this map is the set of inflatable framed curves $(\gamma,V,r)$ with the normalizing condition 
$$
\int_I r \; \mathrm{d}t = \mathrm{length}(I)
$$
(the normalizing condition is easily adjusted by a slight change to our map). Under this map, the energy functional $\mathrm{E}$ becomes
$$
\mathrm{E}(\gamma,V,r) = \int_I (r \kappa)^2 + (r \mathrm{tw})^2 + (D_s r)^2 \; \mathrm{d}s.
$$
The bending and twisting energies are weighted by the radius $r$ with the simple interpretation that bending and twisting a thickened rod requires more energy (cf. a similar weighting approach used in the rod model of \cite{genovese}). The term $(D_s r)^2$ introduces a penalty for changes in radius, under the assumption that the relaxed rod has uniform radius (cf. the stretching term employed in the energy functional of \cite{tunay}).

%%%%%%%%%%%%%%%%%%%%%%%%%%%%%%%%%%%%%%%%%%%%%%%%%%%%%%%%%%%%%%%%%%%
\subsection{Outline of the Paper}
%%%%%%%%%%%%%%%%%%%%%%%%%%%%%%%%%%%%%%%%%%%%%%%%%%%%%%%%%%%%%%%%%%%

Section \ref{sec:quaternionic_coordinates} begins with constructions of quaternionic representations for framed paths and loops. The construction for open framed curves is classical and widely-used.  For closed framed curves, we will use recent work of the author which shows that the space of periodic framed loops corresponds to an infinite-dimensional Grassmann manifold \cite{needham}. We show that, in either case,  the generalized energy functional $\mathrm{E}$ transforms into the Riemannian energy functional on quaternionic path space (Corollary \ref{cor:energy_functional_quaternionic}) and we use this to give a simple expression for the gradient of $\mathrm{E}$ (Corollary \ref{cor:L2_gradient}).

In Section \ref{sec:critical_points}, we use the quaternionic coordinate system to study the periodic critical points of the generalized energy functional. Our first main result (Theorem \ref{thm:E_crit_points_parameterizations}) shows that periodic critical points of $\mathrm{E}$ have explicit parameterizations in terms of trigonometric functions. This is accomplished by using the Grassmannian formalism to transform the variational problem into that of solving a basic ODE. Our second main result (Theorem \ref{thm:1_param_families}) shows that for each critical energy level of $\mathrm{E}$ there is a 1-parameter family of critical points with the same topological behavior as the Ivey-Singer 1-parameter family of Theorem \ref{thm:ivey_singer}. The critical sets of $\mathrm{E}$ are much larger than those of $\mathrm{E}_{Kir}$---they are generically $5$-dimensional rather than $1$-dimensional---and we show in Section \ref{sec:other_knot_types} that there are $\mathrm{E}$-critical framed curves which realize non-torus knot types.

%%%%%%%%%%%%%%%%%%%%%%%%%%%%%%%%%%%%%%%%%%%%%%%%%%%%%%%%%%%%%%%%%%%
\subsection{Notation}
%%%%%%%%%%%%%%%%%%%%%%%%%%%%%%%%%%%%%%%%%%%%%%%%%%%%%%%%%%%%%%%%%%%

For the infinite-dimensional spaces in this paper we work in the Nash-Moser category of tame Fr\'{e}chet spaces, with \cite{hamilton} serving as our main general reference. For a finite-dimensional manifold $M$, we use the notation $\mathcal{P}M=C^\infty(I,M)$ for the \emph{path space of $M$}, where $I=[0,2]$ is fixed (the choice of $[0,2]$ is arbitrary, but makes calculations cleaner). We denote the \emph{loop space of $M$} by $\mathcal{L}M = C^\infty(S^1,M)$ and we identify $S^1$ with the quotient $[0,2]/0\sim 2$. Under this identification, it is convenient to think of $\mathcal{L}M$ as the infinite-codimension submanifold of $\mathcal{P}M$ containing paths which smoothly close.

We use $\mathbb{H}=\mathrm{span}_\R\{1,\qi,\qj,\qk\}$ to denote the quaternions. For $q \in \mathbb{H}$, we use $\overline{q}$ to denote its quaternionic conjugate and $|q|=(q\overline{q})^{1/2}$ to denote its magnitude. We let $\mathh^\ast$ denote $\mathh \setminus \{0\}$. Similar notation will be used for the complex numbers $\C$.

%%%%%%%%%%%%%%%%%%%%%%%%%%%%%%%%%%%%%%%%%%%%%%%%%%%%%%%%%%%%%%%%%%%
\section{Quaternionic Coordinates for Framed Curves}\label{sec:quaternionic_coordinates}
%%%%%%%%%%%%%%%%%%%%%%%%%%%%%%%%%%%%%%%%%%%%%%%%%%%%%%%%%%%%%%%%%%%

%%%%%%%%%%%%%%%%%%%%%%%%%%%%%%%%%%%%%%%%%%%%%%%%%%%%%%%%%%%%%%%%%%%
\subsection{Quaternionic Coordinates for Open Curves}
%%%%%%%%%%%%%%%%%%%%%%%%%%%%%%%%%%%%%%%%%%%%%%%%%%%%%%%%%%%%%%%%%%%

A common trick in elastic rod theory is to represent a framed curve $(\gamma,V)$ as a path in the quaternions $q \in \mathcal{P}\mathh$ which is unique up to a global choice of sign. This representation is facilitated by the existence of a double cover 
$$
\mathrm{H}:\mathcal{P}\mathbb{H}^\ast \rightarrow \{\mbox{framed paths } (\gamma,V)\}/\{\mbox{translations}\},
$$
where we take the quotient of the space of framed paths by the action of $\R^3$ by rigid translations. The map is given explicitly by the formula
\begin{equation}\label{eqn:frame_hopf_map}
\mathrm{H}(q)=(\gamma,V)=\left(\int \overline{q}\textbf{i}q \; \mathrm{d}t, \frac{\overline{q}\textbf{j} q}{|q|^2} \right).
\end{equation}
The integral in the formula denotes a choice of antiderivative. Then $\gamma'=\overline{q}\textbf{i}q$ is well-defined and the need to consider framed paths up to translation is manifested in the fact that we must choose an antiderivative to obtain $\gamma$. To interpret this map, note that $\overline{q}\textbf{i} q$ and $\overline{q}\textbf{j}q$ are paths in the purely imaginary quaternions $\mathrm{Im}(\mathh) \approx \R^3$ and that, as defined, $V$ is a unit normal vector field to $\gamma$. It is also easy to see that this map is a double-covering with the property that $\mathrm{H}(q_1)=\mathrm{H}(q_2)$ if and only if $q_1 = \pm q_2$.

We refer to $\mathrm{H}$ as the \emph{frame-Hopf map}, in reference to its relationship to the well-known homomorphic double covering $S^3 \approx \mathrm{SU}(2) \rightarrow \mathrm{SO}(3)$ and hence to the Hopf fibration $S^1 \hookrightarrow S^3 \rightarrow S^2$ (see, e.g., \cite[Section I.1.4]{gelfand}). Each space in \eqref{eqn:frame_hopf_map} is a tame Fr\'{e}chet space and $\mathrm{H}$ is smooth---refer to \cite{hamilton} for definitions and \cite{needham} for a description of the Fr\'{e}chet space structure of framed path space.

%%%%%%%%%%%%%%%%%%%%%%%%%%%%%%%%%%%%%%%%%%%%%%%%%%%%%%%%%%%%%%%%%%%
\subsection{The Moduli Space of Framed Loops}
%%%%%%%%%%%%%%%%%%%%%%%%%%%%%%%%%%%%%%%%%%%%%%%%%%%%%%%%%%%%%%%%%%%

Since the goal of this paper is to study the knot types of periodic framed curves which are critical with respect to generalized Kirchhoff energy $\mathrm{E}$, it will be useful to quotient by certain ``shape-preserving" group actions on the space of closed framed curves. These are the action of the positive real numbers $\R_{>0}$ by uniform scaling of the curve, the action of $\mathrm{SO}(3)$ by rotations and the action of $S^1$ by \emph{global frame twisting}. This $S^1$-action on $(\gamma,V)$ rotates every frame vector $V$ by the same constant angle in the plane normal to the base curve $\gamma$. We introduce the \emph{moduli space of framed loops}
$$
\mathcal{M}=\left\{\begin{array}{c}
\mbox{closed} \\
\mbox{framed curves} \end{array} \right\}/\left\{\begin{array}{c}
\mbox{translation, scaling,} \\
\mbox{rotation, global frame twisting}\end{array}\right\}.
$$
It will be useful to realize the $\R_{>0}$-quotient by taking a global cross-section consisting of framed curves $(\gamma,V)$ with $\gamma$ of some fixed length. A computationally convenient choice will be to fix $\mathrm{length}(\gamma)=\mathrm{length}(I)=2$. The energy functional $\mathrm{E}$ restricts to the submanifold of fixed-length curves. Moreover, $\mathrm{E}$ is invariant under translations, rotations and global frame twists, so it induces a well-defined functional $\mathrm{E}_\mathcal{M}:\mathcal{M} \rightarrow \R$. The main results of this paper are concerned with the critical points of this induced map.

%%%%%%%%%%%%%%%%%%%%%%%%%%%%%%%%%%%%%%%%%%%%%%%%%%%%%%%%%%%%%%%%%%%
\subsection{Quaternionic Coordinates for $\mathcal{M}$}\label{sec:closed_curves}
%%%%%%%%%%%%%%%%%%%%%%%%%%%%%%%%%%%%%%%%%%%%%%%%%%%%%%%%%%%%%%%%%%%

In this subsection we will show that the moduli space of framed loops $\mathcal{M}$ has a convenient representation in quaternionic coordinates. The construction was first given by the author in \cite{needham} and can be viewed as an infinite-dimensional version of a classical result of Hausmann and Knutson for polygon spaces \cite{hausmann}. We only sketch the construction here; please see \cite{needham} for details.

Let $(\gamma,V)$ be a smoothly closed framed curve and let $q$ be one of its quaternionic representations. One might initially guess that $(\gamma,V)$ is closed if and only if $q$ is a closed quaternionic curve, but this is not the case. The closure of $q$ is not sufficient; for example, the constant curve $q(t) \equiv 1$ maps to an open framed curve. The closure of $q$ is also not necessary; for example, the open quaternionic curve
\begin{equation}\label{eqn:anti_closed_example}
q(t)=\cos(\pi t/2) + \textbf{i} \sin(\pi t/2) + \textbf{j} \cos(\pi t/2) - \textbf{k} \sin(\pi t/2), \;\; t \in [0,2],
\end{equation}
maps to a smoothly closed framed curve under the frame-Hopf map.

To account for the phenomenon illustrated by \eqref{eqn:anti_closed_example}, we introduce the \emph{anti-loop space} of the quaternions
$$
\mathcal{A}\mathh = \{q \in \mathcal{P}\mathh \mid q^{(k)}(0)=-q^{(k)}(2) \; \forall \; k=0,1,2,\ldots\}.
$$
Elements of $\mathcal{A}\mathh$ are referred to as \emph{anticlosed}. A necessary condition for a framed curve $(\gamma,V)$ to be closed is that its quaternionic representation is either closed or anticlosed. This is because the frame-Hopf map is built from the universal covering $\mathrm{SU}(2) \xrightarrow{\times 2} \mathrm{SO}(3)$. We define the antiloop spaces $\mathcal{A}\C$ and $\mathcal{A}\C^2$ analogously to the definition of $\mathcal{A}\mathh$. 

To describe a sufficient condition on $q$ for $(\gamma,V)$ to be closed, we will use the identification $\mathh \approx \C^2$ given by $q=z+w\textbf{j} \leftrightarrow (z,w)$, where the complex $i$ is identified with the quaternionic $\textbf{i}$. The spaces $\mathcal{P}\mathh$ and $\mathcal{P}\C^2$ will be used essentially interchangeably. In these coordinates, the formula for the frame-Hopf map $\mathrm{H}(z,w)=(\gamma,V)$ is
\begin{align}
\gamma &= \int \left(|z|^2-|w|^2,2\mathrm{Im}(z\overline{w}),2\mathrm{Re}(z\overline{w})\right)\;\mathrm{d}t, \label{eqn:hopf_map_complex_coords} \\
V &= \frac{1}{(|z|^2 + |w|^2)}\left(2\mathrm{Im}(zw),\mathrm{Re}(z^2+w^2),\mathrm{Im}(-z^2+w^2)\right),
\end{align}
where the integral symbol in \eqref{eqn:hopf_map_complex_coords} continues to denote a choice of antiderivative. When the choice needs to be made concrete, we will take the antiderivative with initial condition $\gamma(0)=\vec{0}$, so that elements of $\mathcal{M}$ can be thought of as $\mathrm{SO}(3) \times S^1$-orbits of loops based at $\vec{0}$ with fixed length $2$.

The Hermitian $L^2$ inner product on $\mathcal{P}\C$ is denoted
$$
\left<z,w\right>_{L^2} = \int_I z\overline{w} \; \mathrm{d}t
$$
and its induced norm is denoted $\|\cdot\|_{L^2}$. We will use the same notation for the restrictions of these structures to the subspaces $\mathcal{L}\C$ and $\mathcal{A}\C$ as well as for the induced product structures on $\mathcal{P}\C^2 \approx \mathcal{P}\mathh$ and its relevant subspaces, as the meaning of the notation should always be clear from context.

We have the following simple but useful lemma. The first part is a mild generalization of \cite[Lemma 3.10]{needham} and will be useful in this form later in the paper. Both parts of the lemma are proved by elementary computations using   \eqref{eqn:hopf_map_complex_coords}.

\begin{lem}\label{lem:fundamental_lemma}
Let $q=(z,w) \in \mathcal{P}\mathh^\ast$ with $\mathrm{H}(q)=(\gamma,V)$.
\begin{itemize}
\item[(a)] Let $t_0, t_1 \in I$ with $t_0 < t_1$. Then $(\gamma(t_0),V(t_0))=(\gamma(t_1),V(t_1))$ if and only if 
\begin{equation}\label{eqn:self_intersection_condition}
\int_{t_0}^{t_1} |z|^2-|w|^2 \; \mathrm{d}t = \int_{t_0}^{t_1} z \overline{w} \; \mathrm{d}t = 0.
\end{equation}
It follows that $(\gamma,V)$ is a closed framed curve if and only if $q$ is smoothly closed or anticlosed and $z$ and $w$ are $L^2$-equinorm and orthogonal.
\item[(b)] The length of $\gamma$ is given by $\|q\|^2_{L^2}=\|z\|_{L^2}^2+\|w\|_{L^2}^2$. 
\end{itemize}
\end{lem}

The lemma implies that the space of closed framed curves is disconnected, since $\mathcal{L}\mathh^\ast \sqcup \mathcal{A}\mathh^\ast$ forms a disconnected subspace of $\mathcal{P}\mathh^\ast$. Indeed, framed loop space has two path components. The path component of $(\gamma,V)$ with embedded $\gamma$ is determined by the parity of the linking number of $\gamma$ and $\gamma + \epsilon V$ for small $\epsilon$. The path component of a framed curve with nonembedded base curve is determined by continuity (see \cite[Section 3.2.1]{needham2}).

For the rest of the paper, we will use the notation $\mathcal{V}$ to stand for either $\mathcal{L}\C$ or $\mathcal{A}\C$. We define the \emph{Stiefel manifold} or Hermitian-orthonormal 2-frames
$$
\mathrm{St}_2(\mathcal{V}) = \left\{(z,w) \in \mathcal{V}^2 \mid \left<z,w\right>_{L^2}=0, \; \|z\|_{L^2}=\|w\|_{L^2} = 1 \right\}
$$
in direct analogy with the finite-dimensional version. We will treat the Stiefel manifold as a codimension-4  (real) submanifold of $\mathcal{V}^2$ with a weak Riemannian structure induced by the real part of the $L^2$ inner product. We also define the dense open submanifold
$$
\mathrm{St}_2^\circ(\mathcal{V}) = \{(z,w) \in \St \mid (z(t),w(t)) \neq 0 \; \forall \; t \in I\}.
$$
The group $\mathrm{U}(2)$ acts freely on $\St$ by pointwise multiplication and we define the \emph{Grassmann manifold} to be the quotient $\Gr=\St/\mathrm{U}(2)$. We also define $\mathrm{Gr}_2^\circ(\mathcal{V}) = \mathrm{St}_2^\circ (\mathcal{V})/\mathrm{U}(2)$. The Grassmannian inherits a weak Riemannian structure from the $\mathrm{U}(2)$-invariant Riemannian metric on $\St$.

\begin{thm}{\cite[Theorem 3.15]{needham}}\label{thm:grassmannian}
The frame-Hopf map induces a diffeomorphism 
$$
\mathrm{Gr}^\circ_2(\mathcal{L}\C) \sqcup \mathrm{Gr}^\circ_2(\mathcal{A}\C) \approx \mathcal{M}.
$$ 
\end{thm}

To prove the theorem, we use Lemma \ref{lem:fundamental_lemma} to show that the disjoint union $\mathrm{St}^\circ_2(\mathcal{L}\C) \sqcup \mathrm{St}_2^\circ(\mathcal{A}\C)$ double-covers the space of fixed-length closed curves. Next we show that frame-Hopf map is equivariant with respect to the $\mathrm{U}(2)$-action in quaternionic coordinates and the $\mathrm{SO}(3) \times S^1$-action on framed loop space by rotations and global frame twists.

%%%%%%%%%%%%%%%%%%%%%%%%%%%%%%%%%%%%%%%%%%%%%%%%%%%%%%%%%%%%%%%%%%%
\subsection{The Energy Functional in Quaternionic Coordinates}\label{sec:energy_functional_quaternionic}
%%%%%%%%%%%%%%%%%%%%%%%%%%%%%%%%%%%%%%%%%%%%%%%%%%%%%%%%%%%%%%%%%%%

The geometry of a framed path $(\gamma,V)$ is described by its \emph{curvatures},
$$
\kappa_1=\kappa_1(\gamma,V)=\left<D_s^2 \gamma, V\right> \;\; \mbox{ and } \;\; \kappa_2=\kappa_2(\gamma,V) = \left<D_s^2 \gamma, D_s \gamma \times V\right>
$$
and by its twist rate $\mathrm{tw}$ defined in \eqref{eqn:curvature_and_twist}. We also consider the stretch rate $\mathrm{st}$ defined in \eqref{eqn:stretch_rate}. We wish to describe these invariants in quaternionic coordinates. Let $\left<\cdot,\cdot\right>_\mathbb{H}$ denote the Euclidean inner product on $\mathbb{H} \approx \R^4$, which is expressed in quaternionic notation as $\left<p,q\right>_\mathh = \frac{1}{2}(p \overline{q} + q\overline{p})$.  We use $\left<\cdot,\cdot\right>_{\C^2}$ to denote the standard Hermitian inner product on $\C^2$.

\begin{lem}\label{lem:darboux_curvatures}
Let $(\gamma,V)$ be a framed curve and $q \in \mathcal{P}\mathbb{H}^\ast$ one of its associated quaternionic paths. Then the geometric invariants of $(\gamma,V)$ are given by
\begin{equation}\label{eqn:quaternionic_invariants}
\kappa_1 = -2 \frac{\left<q',\textbf{k}q\right>_\mathh}{|q|^4}, \;\;\; \kappa_2 = -2 \frac{\left<q',\textbf{j}q\right>_\mathh}{|q|^4}, \;\;\; \mathrm{tw} = -2 \frac{\left<q',\textbf{i}q\right>_\mathh}{|q|^4}, \;\;\; \mathrm{st} = 2 \frac{\left<q',q\right>_\mathh}{|q|^4}.
\end{equation}
These are given in complex coordinates $q=(z,w)$ by
\begin{equation}\label{eqn:complex_invariants}
\begin{array}{ll}
\displaystyle \kappa_1 = -2 \frac{\mathrm{Im}\left<(z',w'),(-\overline{w},\overline{z})\right>_{\C^2}}{(|z|^2+|w|^2)^2}, &\hspace{.2in} \displaystyle \kappa_2 = -2 \frac{\mathrm{Re}\left<(z',w'),(-\overline{w},\overline{z})\right>_{\C^2}}{(|z|^2+|w|^2)^2} \\
& \\
\displaystyle \mathrm{tw} = -2 \frac{\mathrm{Im}\left<(z',w'),(z,w)\right>_{\C^2}}{(|z|^2+|w|^2)^2}, &\hspace{.2in}  \displaystyle \mathrm{st} = 2 \frac{\mathrm{Re}\left<(z',w'),(z,w)\right>_{\C^2}}{(|z|^2+|w|^2)^2}.\end{array}
\end{equation}
\end{lem}

\begin{proof}
The formulas \eqref{eqn:quaternionic_invariants} for the invariants $\kappa_1$, $\kappa_2$ and $\mathrm{tw}$ are well-known, although typically expressed slightly differently (see, e.g., \cite[Section 2.6]{dichmann}, where the minor discrepancies with our formulas are due to small differences in the map used to obtain a framed curve from a quaternionic path). The formula for $\mathrm{st}$ appears in \cite{tunay}, but is less well-known. To derive it, first note that $\|\gamma'\| = |q|^2$ holds pointwise. Then
$$
\mathrm{st}=\frac{1}{\|\gamma'\|^2} \frac{d}{dt} \|\gamma'\| = \frac{1}{|q|^4} \frac{d}{dt} q \overline{q} =  \frac{1}{|q|^4} \left(q' \overline{q} + q \overline{q'}\right) =  \frac{1}{|q|^4} 2\left<q',q\right>_\mathh.
$$
The complex formulas \eqref{eqn:complex_invariants} can be derived by writing $q=z+w\textbf{j}$ and using quaternion and complex arithmetic and the obvious isomorphism 
\begin{equation}\label{eqn:inner_product_iso}
(\mathh,\left<\cdot,\cdot\right>_\mathh) \approx (\C^2, \mathrm{Re} \left<\cdot, \cdot \right>_{\C^2})
\end{equation}
of inner product spaces.
\end{proof}

\begin{cor}\label{cor:energy_functional_quaternionic}
The generalized energy of a framed curve $(\gamma,V)$ is given in terms of its quaternionic representative $q$ by
$$
\mathrm{E}(\gamma,V)=4 \int_I |q'|^2 \; \mathrm{d}t.
$$
\end{cor}

\begin{proof}
Using the fact that the curvatures $\kappa_1, \kappa_2$ of a framed curve $(\gamma,V)$ are related to the curvature $\kappa$ of $\gamma$ by the formula $\kappa^2=\kappa_1^2+\kappa_2^2$ and the definition $\mathrm{d}s=\|\gamma'\|\mathrm{d}t$, the energy can be expressed as
$$ 
\mathrm{E}(\gamma,V) =\int_I (\kappa_1^2 + \kappa_2^2 + \mathrm{tw}^2 + \mathrm{st}^2)\|\gamma'\|^3 \; \mathrm{d}t. 
$$
Lemma \ref{lem:darboux_curvatures} implies that
$$
\kappa_1^2+\kappa_2^2=\frac{4}{|q|^8} \mathrm{Im}(-z'w + w'z)^2 + \mathrm{Re}(-z'w+w'z)^2 = \frac{4}{|q|^8} |-z'w + w'z|^2
$$
and similarly that $\mathrm{tw}^2 + \mathrm{st}^2 = 4 |z'\overline{z} + w'\overline{z}|^2/ |q|^8$. We then simplify
$$
|-z'w + w'z|^2 +  |z'\overline{z} + w'\overline{z}|^2 = (|z|^2+|w|^2)^2(|z'|^2+|w'|^2)^2=|q|^2 |q'|^2
$$
so that, using $\|\gamma'\|=|q|^2$, we have
$$
\mathrm{E}(\gamma,V)= \int_I \frac{4}{|q|^8} \cdot |q|^2 |q'|^2 \cdot |q|^6 \; \mathrm{d}t = 4 \int_I |q'|^2 \; \mathrm{d}t.
$$
\end{proof}

If we consider Kirchhoff energy $\mathrm{E}_{Kir}$ as the restriction of $\mathrm{E}$ to the submanifold of arclength-parameterized framed curves, then its quaternionic representative keeps the same form but is viewed as an energy functional on the submanifold $\mathcal{P}S^3 \subset \mathcal{P}\mathh^\ast$. This idea was essentially used in \cite[Section II]{hu}.

\begin{cor}\label{cor:L2_gradient}
The gradient of the restriction of $\mathrm{E}$ to $\mathcal{L}\mathh^\ast \sqcup \mathcal{A}\mathh^\ast$ with respect to  $\mathrm{Re}\left<\cdot,\cdot\right>_{L^2}$ is $\mathrm{grad} \, \mathrm{E} |_q = -8 q''$
\end{cor}

\begin{proof}
A standard variational calculation shows that the derivative of $\mathrm{E}$ at $q \in \mathcal{L}\mathh^\ast \sqcup \mathcal{A}\mathh^\ast$ in the direction $Q \in T_q \left(\mathcal{L}\mathh^\ast \sqcup \mathcal{A}\mathh^\ast\right) \approx \mathcal{L}\mathh \cup \mathcal{A}\mathh$ is given by
$$
D\mathrm{E}(q)(Q)= 4 \int_I q' \overline{Q'} + Q' \overline{q'} \; \mathrm{d}t = 8 \mathrm{Re} \int_I \left<q', Q' \right>_{\C^2} \; \mathrm{d}t,
$$
where the second equality uses the isomorphism \eqref{eqn:inner_product_iso}. Integrating by parts under the assumption that $q$ and $Q$ are both closed or anticlosed, we have
$$
D\mathrm{E}(q)(Q) = -8\mathrm{Re}\int_I \left<q'',Q\right>_{\C^2} \; \mathrm{d}t = \mathrm{Re} \left<-8 q'', Q \right>_{L^2}.
$$
\end{proof}

\begin{remark}
It may seem arbitrary to use the $L^2$ metric on $\mathrm{P}\C^2$ as a tool for calculations about framed paths. We show in \cite{needham} that the $L^2$ metric on $\mathcal{P}\C$ is the pullback under $\mathrm{H}$ of a natural metric on framed path space. The framed path space metric is closely related to the \emph{elastic metrics} used on spaces of plane curves which have generated a lot of recent interest in the field of computer vision. Refer to \cite{bauer} for a survey of recent work on this Riemannian approach to curve matching problems.
\end{remark}

%%%%%%%%%%%%%%%%%%%%%%%%%%%%%%%%%%%%%%%%%%%%%%%%%%%%%%%%%%%%%%%%%%%
\section{Critical Points of $\mathrm{E}_\mathcal{M}$}\label{sec:critical_points}
%%%%%%%%%%%%%%%%%%%%%%%%%%%%%%%%%%%%%%%%%%%%%%%%%%%%%%%%%%%%%%%%%%%

\subsection{Main Theorem}

We now focus on the induced map $\mathrm{E}_\mathcal{M}$ on the moduli space of closed framed curves $\mathcal{M}$ in order to study the knot types which occur as critical points. The critical points of $\mathrm{E}_\mathcal{M}$ have a surprisingly simply form when written in quaternionic/complex coordinates. We define a set of quaternionic curves $\mathfrak{Q}$ as follows. For $\xi_1,\xi_2,\zeta_1,\zeta_2 \in \C$ and $c,d \in \Z$, let
\begin{equation}\label{eqn:solution_quaternionic_curves}
q(c,d,\xi_1,\xi_2,\zeta_1,\zeta_2) = \left(\begin{array}{c}
z(c,\xi_1,\xi_2) \\
w(d,\zeta_1,\zeta_2) \end{array}\right) = \left(\begin{array}{c}
\xi_1 \mathrm{e}(c) + \xi_2 \mathrm{e}( -c) \\
\zeta_1 \mathrm{e}(d) + \zeta_2 \mathrm{e}(-d)
\end{array}\right),
\end{equation}
where we use the shorthand notation $\mathrm{e}(\lambda)$ for the function $t \mapsto \exp (i\pi \lambda t /2)$. Let $\mathfrak{Q}=\mathfrak{Q}_1 \sqcup \mathfrak{Q}_2$, where the $\mathfrak{Q}_j$ contain quaternionic curves of the form \eqref{eqn:solution_quaternionic_curves} with particular parameter choices. The first set is defined simply by
$$
\mathfrak{Q}_1 = \{q(c,c,1,0,0,1) \mid c \neq 0 \}
$$
The set $\mathfrak{Q}_2$ is more complicated; it contains $q(c,d,\xi_1,\xi_2,\zeta_1,\zeta_2)$ such that all of the following conditions hold:
\begin{enumerate}
\item $c=d \; \mathrm{mod} \, 2$,
\item $c > d \geq 0$,
\item $\xi_1$ is a real number greater than or equal to zero,
\item $|\xi_1|^2 + |\xi_2|^2 = |\zeta_1|^2 + |\zeta_2|^2 = 1$.
\end{enumerate}

\begin{thm}\label{thm:E_crit_points_parameterizations}
The critical points of the generalized Kirchhoff energy $\mathrm{E}_\mathcal{M}$ can be expressed as quaternionic curves of the form \eqref{eqn:solution_quaternionic_curves}. In particular, the critical points  are in bijective correspondence with the quaternionic parameterizations in $\mathfrak{Q}$.
\end{thm}

The proof will use the following lemma.

\begin{lem}\label{lem:normal_spaces}
Let $q=(z,w) \in \mathrm{St}_2(\mathcal{V})$. The normal space to $T_q \St$ in $T_q \mathcal{V}^2$ has orthonormal basis
\begin{equation}\label{eqn:stiefel_normal_space}
(z, 0), \; (0,w), \; \frac{1}{\sqrt{2}}(w,z),\; \frac{1}{\sqrt{2}}(-i w, i z)
\end{equation}
with respect to $\mathrm{Re}\left<\cdot,\cdot\right>_{L^2}$.
\end{lem}

\begin{proof}
One can show that $T_q \St \subset T_q\mathcal{V}^2 \approx \mathcal{V}^2$ is equal to
\begin{equation}\label{eqn:stiefel_tangent_space}
 \left\{(Z,W) \in \mathcal{V}^2 \mid  \mathrm{Re}\left<z,Z\right>_{L^2} = \mathrm{Re}\left<w,W\right>_{L^2} = \left<z,W\right>_{L^2} + \left<Z,w\right>_{L^2}  =0\right\}.
\end{equation}
This is a straightforward generalization of the finite-dimensional real case \cite[Section 2.2.1]{edelman}. In this infinite-dimensional setting, we note that $\St$ is the level set at $\vec{0}$ for the map $\mathcal{V}^2 \rightarrow \R^4$ defined by
$$
(z,w) \mapsto \left(\|z\|_{L^2}^2-1,\|w\|_{L^2}^2-1,\mathrm{Re}\left<z,w\right>_{L^2},\mathrm{Im}\left<z,w\right>_{L^2}\right).
$$
The kernel of the derivative of this map is exactly the space \eqref{eqn:stiefel_tangent_space}. The claim then follows by applying \cite[Section III, Theorem 2.3.1]{hamilton}, which is an implicit function theorem for maps from a tame Fr\'{e}chet manifold to a finite-dimensional vector space. A straightforward calculation shows that each vector in the list \eqref{eqn:stiefel_normal_space} is $L^2$-orthogonal to $T_q \St$ and that the vectors are $L^2$-orthonormal.
\end{proof}

\begin{proof}[Proof of Theorem 1]
Applying Theorem \ref{thm:grassmannian} and Corollary \ref{cor:L2_gradient}, we can rephrase the problem as a search for the critical points of the map $\mathrm{Gr}_2^\circ(\mathcal{V}) \rightarrow \R$ given by
\begin{equation}\label{eqn:grassmannian_energy}
[q] \mapsto \int_I |q'|^2 \;\mathrm{d}t,
\end{equation}
where $[q]$ denotes the $\mathrm{U}(2)$-orbit of $q \in \mathrm{St}_2^\circ(\mathcal{V})$. The Grassmannian $\Gr$ can be treated locally as a codimension-8 submanifold of $\mathcal{V}^2$ since $\St$ is a codimension-4 submanifold of $\mathcal{V}^2$ and we have local slice charts to the quotient map $\St \rightarrow \St/\mathrm{U}(2)=\Gr$. Our task is therefore to determine those $q$ lying in this submanifold such that the $L^2$-gradient $\mathrm{grad} \, \mathrm{E}|_q$ lies in the 8-dimensional normal space to $\Gr$. The energy functional is invariant under the action of $\mathrm{U}(2)$ (whence the map \eqref{eqn:grassmannian_energy} is well-defined), and it follows that $\mathrm{grad} \, \mathrm{E}|_q$ is always orthogonal to the $\mathrm{U}(2)$-orbits. We therefore wish to determine the $\mathrm{U}(2)$-orbits of points $q=(z,w) \in \St$ with $\mathrm{grad} \, \mathrm{E}|_q$ in the $4$-dimensional normal direction to the Stiefel manifold. 

Using Corollary \ref{cor:L2_gradient} and Lemma \ref{lem:normal_spaces}, we see that the critical points of $\mathrm{E}$ must satisfy
$$
q''=\lambda_1(z,0) + \lambda_2 (0,w) + \lambda_3 (w,z) + \lambda_4 (-i w,i z),
$$
for some scalars $\lambda_j \in \R$. In matrix form,
$$
\left(\begin{array}{c}
z'' \\
w '' \end{array}\right) = \left(\begin{array}{cc}
\lambda_1 & \lambda_3 - i \lambda_4 \\
\lambda_3 + i \lambda_4 &  \lambda_2 \end{array}\right) \left(\begin{array}{c}
z \\
w \end{array}\right).
$$
The Hermitian coefficient matrix is diagonalizable by unitary matrices, and since we are searching for $q$ only up to the action of $\mathrm{U}(2)$, we can rename parameters and replace the system by
\begin{equation}\label{eqn:crit_points_of_E}
\left(\begin{array}{c}
z'' \\
w '' \end{array}\right) = \left(\begin{array}{cc}
\lambda_1 & 0\\
0 & \lambda_2 \end{array}\right) \left(\begin{array}{c}
z \\
w \end{array}\right).
\end{equation}
Solutions to the system \eqref{eqn:crit_points_of_E} are of the form
\begin{equation}\label{eqn:solutions}
\left(\begin{array}{c}
z(t) \\
w(t)
\end{array}\right) = \left(\begin{array}{c}
\xi_1 e^{t\sqrt{\lambda_1}} + \xi_2 e^{-t\sqrt{\lambda_1}} \\
\zeta_1 e^{t\sqrt{\lambda_2}} + \zeta_2 e^{-t\sqrt{\lambda_2}}
\end{array}\right),
\end{equation}
where $\xi_1,\xi_2,\zeta_1,\zeta_2 \in \C$ are constants of integration.  

There are two remaining tasks to complete the proof. We first need to determine conditions which guarantee that a curve of the form \eqref{eqn:solutions} lies in $\St$. Second, we need to normalize coefficients to get unique representatives over $\mathrm{U}(2)$-orbits. These tasks are divided into three claims.

\smallskip
\noindent \textbf{Claim 1.} A solution of the form \eqref{eqn:solutions} lies in $\mathcal{L}\mathbb{H} \cup \mathcal{A}\mathh$  if and only if $\sqrt{\lambda_1}=ic\pi/2$ and $\sqrt{\lambda_2}=id\pi/2$ for some $c,d \in \Z$ with $c = d \; \mathrm{mod} \, 2$; that is, it can be written in the form \eqref{eqn:solution_quaternionic_curves} with $c=d \; \mathrm{mod} \, 2$. Moreover, such a solution lies in $\St$ if and only if $c$ and $d$ are not both zero  and 
$$
\left\{\begin{array}{cl}
|\xi_1|^2+|\xi_2|^2 = |\zeta_1|^2 + |\zeta_2|^2 = 1 & \mbox{ if } c \neq \pm d \\
(\xi_1,\xi_2) \mbox{ and } (\zeta_1, \zeta_2) \mbox{ are orthonormal in $\C^2$} & \mbox{ if } c= d \\
(\xi_1,\xi_2) \mbox{ and } (\zeta_2,\zeta_1) \mbox{ are orthonormal in $\C^2$} & \mbox{ if } c=-d.  \end{array}\right.
$$
\smallskip

The first statement of the claim is clear. Moreover, it is also clear that if $c=d=0$, then the  solution cannot lie in $\St$. Consider $q(c,d,\xi_1,\xi_2,\zeta_1,\zeta_2)=q=(z,w)$ with $c=d \; \mathrm{mod} \, 2$ and $c,d$ not both zero. Using $L^2$-orthonormality of the functions $\mathrm{e}(k)$, $k \in \Z$, one is able to show that 
$$
\|z\|^2_{L^2}=|\xi_1|^2 + |\xi_2|^2, \;\;\;\; \|w\|^2_{L^2} = |\zeta_1|^2 + |\zeta_2|^2 \;\; \mbox{and} \;\; \left<z,w\right>_{L^2} = \left\{
\begin{array}{cl}
0 & \mbox{ if } c \neq \pm d \\
\left<(\xi_1,\xi_2),(\zeta_1,\zeta_2)\right>_{\C^2} & \mbox{ if } c= d \\
\left<(\xi_1,\xi_2),(\zeta_2,\zeta_1)\right>_{\C^2}& \mbox{ if } c=-d. 
\end{array}\right. 
$$
Therefore the solution $q$ lies in $\St$ if and only if the conditions of the claim are satisfied.

\smallskip
\noindent \textbf{Claim 2.} The set $\mathfrak{Q}_1$ is a $\mathrm{U}(2)$-cross-section of the set $\widetilde{\mathfrak{Q}}_1=\{q(c,\pm c,\xi_1,\xi_2,\zeta_1,\zeta_2) \mid c \neq 0\} \cap \St$.
\smallskip

First note that, by renaming parameters, an arbitrary element of $\widetilde{\mathfrak{Q}}_1$ can always be written in the form $q=q(c,c,\xi_1,\xi_2,\zeta_1,\zeta_2)$. Now the claim follows by the fact that for any such $q$, there is a unique $A \in \mathrm{U}(2)$ such that $q=(\mathrm{e}(c),\mathrm{e}(-c))\cdot A$; namely,
$$
q=(\mathrm{e}(c),\mathrm{e}(-c)) \cdot \left(\begin{array}{cc}
\xi_1 & \zeta_1 \\
\xi_2 & \zeta_2
\end{array}\right).
$$
It is also easy to see that $(\mathrm{e}(c),\mathrm{e}(-c)) \cdot A \in \mathfrak{Q}_1$ if and only $A$ is the identity. This completes the proof of Claim 2.

\smallskip
\noindent \textbf{Claim 3.} The set $\mathfrak{Q}_2$ is a $\mathrm{U}(2)$-cross-section of the set $\widetilde{\mathfrak{Q}}_2=\{q(c,d,\xi_1,\xi_2,\zeta_1,\zeta_2) \mid c \neq \pm d \} \cap \St$.
\smallskip

Renaming parameters, we can assume that $c, d \geq 0$. We first  show that if $q(c,d,\xi_1,\xi_2,\zeta_1,\zeta_2)=q=(z,w) \in \mathfrak{Q}_2$ is taken to $\mathfrak{Q}_2$ by some $A \in \mathrm{U}(2)$, then $A$ is the identity matrix.  Writing
$$
A=e^{i\theta} \left(\begin{array}{cc}
u & v \\
-\overline{v} & \overline{u}\end{array}\right), 
$$
we see that the first coordinate of $q \cdot A$ is given by 
$$
e^{i\theta}\left(u(\xi_1 \mathrm{e}(c) + \xi_2 \mathrm{e}(-c)) - \overline{v}(\zeta_1 \mathrm{e}(d) + \zeta_2 (\mathrm{e}(-d)))\right).
$$
It follows by the $L^2$-orthonormality of the functions $\mathrm{e}(k)$ that for this to be the first coordinate of an element of $\mathfrak{Q}_2$, it must be that $u=0$ or $v=0$. The $u=0$ case is ruled out, as this would imply that $A$ switches the entries of $q$ and this contradicts condition (2) in the definition of $\mathfrak{Q}_2$. If $v=0$, then condition (3) implies that $e^{i\theta} u$ is a positive real number and condition (4) implies that $|e^{i\theta} u|=1$. It follows that $A$ is the identity matrix.

It is therefore sufficient to show that for any find any $q(c,d,\xi_1,\xi_2,\zeta_1,\zeta_2)=q=(z,w) \in \widetilde{\mathfrak{Q}}_2$, there exists $A \in \mathrm{U}(2)$ such that $q \cdot A \in \mathfrak{Q}_2$. Assume without loss of generality (by renaming parameters as necessary) that $c,d \geq 0$. We arrange that $c > d \geq 0$ by switching the coordinate functions of $q$ if necessary, and this corresponds to multiplying by an element of $\mathrm{U}(2)$. If $\xi_1 = 0$ then we are done and otherwise we take the curve to $\mathfrak{Q}_2$ by multiplying by the matrix 
$$
\left(\begin{array}{cc}
\frac{\overline{\xi_1}}{|\xi_1|} & 0 \\
0 & \frac{\xi_1}{|\xi_1|}\end{array}\right) \in \mathrm{U}(2).
$$
This completes the proof of Claim 3 and of the theorem.
\end{proof}

\subsection{Energy Levels for $\mathrm{E}$}

Using Theorem \ref{thm:E_crit_points_parameterizations}, we are able to calculate the possible energy levels that a framed curve can realize. For $q=(z,w) \in \mathfrak{Q}$,
\begin{align*}
|z'(t)|^2 &= \left| \frac{i\pi c}{2} \xi_1 \mathrm{e}(c)\right|^2 + \left|\frac{-i\pi c}{2} \xi_2 \mathrm{e}(-c)\right|^2 +2 \mathrm{Re}\left(\frac{i\pi c}{2} \xi_1 \mathrm{e}(c) \overline{\left(\frac{-i\pi c}{2} \xi_2 \mathrm{e}(-c)\right)}\right)\\
&=\left(\frac{\pi c}{2}\right)^2 - 2\left(\frac{\pi c}{2}\right)^2 \mathrm{Re}\left(\xi_1\overline{\xi_2}\mathrm{e}(2c) \right),
\end{align*}
where we have used $|\xi_1|^2+|\xi_2|^2=1$. Similarly,
$$
|w'(t)|^2=\left(\frac{\pi d}{2}\right)^2 - 2\left(\frac{\pi d}{2}\right)^2 \mathrm{Re}\left(\zeta_1\overline{\zeta_2}\mathrm{e}(2d)\right).
$$
Recalling that $I=[0,2]$ and that $\mathrm{e}(k)$ denotes the function $t \mapsto \exp(i\pi k t/2)$, we have
\begin{align*}
\mathrm{E}(q) &= \int_I |z'|^2 + |w'|^2 \; \mathrm{d}t =\int_I \left(\frac{\pi c}{2}\right)^2 - 2\left(\frac{\pi c}{2}\right)^2 \mathrm{Re}\left(\xi_1\overline{\xi_2}\mathrm{e}(2c) \right) + \left(\frac{\pi d}{2}\right)^2 - 2\left(\frac{\pi d}{2}\right)^2 \mathrm{Re}\left(\zeta_1\overline{\zeta_2}\mathrm{e}(2d)\right) \; \mathrm{d}t \\
&=\frac{\pi^2(c^2+d^2)}{2} - 2\left(\frac{\pi c}{2}\right)^2 \mathrm{Re}\left(  \xi_1\overline{\xi_2} \int_I \mathrm{e}(2c) \; \mathrm{d}t\right)- 2\left(\frac{\pi d}{2}\right)^2 \mathrm{Re}  \left(\zeta_1\overline{\zeta_2} \int_I \mathrm{e}(2d) \; \mathrm{d}t\right) =\frac{\pi^2(c^2+d^2)}{2}.
\end{align*}
We have proved:

\begin{cor}
The possible critical energy levels of $\mathrm{E}_\mathcal{M}$ are $\pi^2(c^2+d^2)/2$ for any integers $c$ and $d$ which satisfy $c=d \; \mathrm{mod} \, 2$ and are not both zero.
\end{cor}

%%%%%%%%%%%%%%%%%%%%%%%%%%%%%%
\subsection{Isolated Critical Points}
%%%%%%%%%%%%%%%%%%%%%%%%%%%%%%

Theorem \ref{thm:E_crit_points_parameterizations} implies that the isolated critical points of $\mathrm{E}_\mathcal{M}$ are exactly those corresponding to elements of the set $\mathfrak{Q}_1$. Straightforward calculations using the formulas for the geometric invariants given in Lemma \ref{lem:darboux_curvatures} show that the periodic parameterized framed curve corresponding to $q(c,c,1,0,0,1) \in \mathfrak{Q}_1$ has $\kappa_1 = \frac{\pi c}{2}$ and $\kappa_2=\mathrm{tw}=\mathrm{st}=0$.  We conclude that the isolated critical points of $\mathrm{E}_\mathcal{M}$ are multiply-covered, arclength-parameterized, untwisted round circles.

%%%%%%%%%%%%%%%%%%%%%%%%%%%%%%
\subsection{One Parameter Families of Critical Points}
%%%%%%%%%%%%%%%%%%%%%%%%%%%%%%

Let $c$ and $d$ be integers which are not both zero. We will consider the simple 1-parameter family $q_u=q(c,d,u,\sqrt{1-u^2},1,0)$ of elements of $\widetilde{\mathfrak{Q}}_1$, whose $\mathrm{U}(2)$-orbits map to critical points of $\mathrm{E}_\mathcal{M}$ under the frame-Hopf map. It will be convenient to introduce the change of variables
$$
h:=\frac{c+d}{2} \;\;\; \mbox{ and } \;\;\; k:=\frac{c-d}{2}.
$$
We will also adopt the shorthand notations $\mathrm{c}(\lambda)$ and $\mathrm{s}(\lambda)$ for the functions $t \mapsto \cos(\lambda \pi t)$ and $t \mapsto \sin(\lambda \pi t)$, respectively. Then, up to a translation, $q_u$ maps to the framed loop $(\gamma_u,V_u)$, where
\begin{equation}\label{eqn:explicit_gamma_u}
\gamma_u = \frac{2}{\pi}\left(\frac{u \sqrt{1-u^2}}{h+k} \mathrm{s}(h+k),  -\frac{u}{k} \mathrm{c}(k) + \frac{\sqrt{1-u^2}}{h} \mathrm{c}(h),  \frac{u}{k} \mathrm{s}(k) + \frac{\sqrt{1-u^2}}{h} \mathrm{s}(h) \right)
\end{equation}
and
$$
V_u(t) = \frac{2}{|q_u|^2} \left(u\mathrm{s}(h) - \sqrt{1-u^2} \mathrm{s}(k), \mathrm{c}(h)\mathrm{c}(k)+ u\sqrt{1-u^2},(1-u^2) \mathrm{s}(h)\mathrm{c}(k) - u^2 \mathrm{c}(h)\mathrm{s}(k)\right),
$$

%%%%If margins need to be increased, uncomment below and delete above.
%\begin{align*}
%&V_u(t) = \frac{2}{|q_u|^2} \left(
%u\mathrm{s}(h) - \sqrt{1-u^2} \mathrm{s}(k), \mathrm{c}(h)\mathrm{c}(k)+ u\sqrt{1-u^2}, \right. \%\
%&\hspace{2.5in} \left. (1-u^2) \mathrm{s}(h)\mathrm{c}(k) - u^2 \mathrm{c}(h)\mathrm{s}(k) %%\right),
%\end{align*}
with
$$
|q_u|^2 = 2 + 2u\sqrt{1-u^2} \mathrm{c}(h+k).
$$

In particular,
$$
\gamma_0(t)=\frac{2}{\pi h} \left(0, \mathrm{c}(h), \mathrm{s}(h) \right) \;\; \mbox{ and } \;\;  V_0(t)= \mathrm{c}(k)\left(0, \mathrm{c}(h), \mathrm{s}(h) \right) + \mathrm{s}(k) \left(1,0,0\right).
$$
Clearly $\gamma_0$ is an arclength-parameterized $|h|$-times-covered round circle. Moreover, the linking number of $\gamma_0$ and $\gamma_0 + \epsilon V_0$ (for $\epsilon$ sufficiently small) is $k$. Similarly,
$$
\gamma_1(t)=\frac{2}{\pi k} \left(0,-\mathrm{c}(k),\mathrm{s}(k)\right) \;\; \mbox{ and } \;\; V_1(t)=.-\mathrm{c}(h)\left(0,-\mathrm{c}(k),\mathrm{s}(k)\right) + \mathrm{s}(h)\left(1,0,0\right),
$$
so that $(\gamma_1,V_1)$ is an arclength-parameterized, $|k|$-covered round circle whose image is linked $-h$ times by $\gamma_1+ \epsilon V_1$.

It remains to determine what happens between the endpoints of the one-parameter family. We first show that the base curve $\gamma_u$ is nonembedded for exactly one value of $u \in (0,1)$.

\begin{cor}
For the 1-parameter family of critical points described above, the only $u \in (0,1)$ for which the base curve $\gamma_u$ is nonembedded is
$$
u= \sqrt{\frac{k^2}{h^2+k^2}}.
$$
\end{cor}

\begin{proof}
Assume that $\gamma_u$ self intersects at some $u$ at parameter values $t_0< t_1$. The 1-parameter family is given in complex coordinates by 
$$
(z_u,w_u) = (u \mathrm{e}(c)+ \sqrt{1-u^2} \mathrm{e}(-c), \mathrm{e}(d)),
$$
so that the condition
$$
 \int_{t_0}^{t_1} z_u \overline{w_u} \; \mathrm{d}t = 0
$$
from Lemma \ref{lem:fundamental_lemma} reads
\begin{align*}
0&= \int_{t_0}^{t_1} u \mathrm{e}(c) \mathrm{e}(-d) + \sqrt{1-u^2}\mathrm{e}(-c) \mathrm{e}(-d) \; \mathrm{d}t = \int_{t_0}^{t_1} u \mathrm{e}(2k) + \sqrt{1-u^2}\mathrm{e}(-2h) \; \mathrm{d}t  \\
&=\frac{u}{ik\pi} \left(\mathrm{e}(2kt_1)-\mathrm{e}(2kt_0) \right) +  \frac{\sqrt{1-u^2}}{-ih\pi} \left(\mathrm{e}(-2ht_1) - \mathrm{e}(-2ht_0)\right).
\end{align*}
Equivalently, 
$$
hu\left(\mathrm{e}(2kt_1)-\mathrm{e}(2kt_0) \right) = k\sqrt{1-u^2} \left(\mathrm{e}(-2ht_1) - \mathrm{e}(-2ht_0)\right).
$$
Taking the squared magnitude of each side yields
$$
h^2u^2(2-2\mathrm{Re}\,\mathrm{e}(2k(t_0-t_1))) = k^2(1-u^2)(2-2\mathrm{Re}\,\mathrm{e}(-2h(t_0-t_1))).
$$

We claim that $\mathrm{Re}\,\mathrm{e}(2k(t_0-t_1))=\mathrm{Re}\,\mathrm{e}(-2h(t_0-t_1))$, or equivalently that $\cos((t_0-t_1)k\pi)=\cos((t_0-t_1)h\pi)$.  In this case we are done, as
$$
h^2u^2=k^2(1-u^2) \; \Rightarrow \; u= \sqrt{\frac{k^2}{h^2+k^2}}.
$$
From the first coordinate of the explicit parameterization \eqref{eqn:explicit_gamma_u} of $\gamma_u$, we see that for $\gamma_u$ to have a self-intersection at parameter values $t_0 < t_1$ it must be that $\sin((h+k)\pi t_0)=\sin((h+k) \pi t_1)$, so a necessary condition is that
$$
t_1=\left\{\begin{array}{l}
\displaystyle t_0+\frac{2j}{h+k}\;\; \mbox{ or }\\
 \\
\displaystyle \frac{2j+1}{h+k} -t_0, \end{array}\right. \;\;\; \mbox{ for } j=0,1,\ldots,h+k-1.
$$
Taking $t_1=t_0+2j/(h+k)$, we have
\begin{align*}
\cos((t_0-t_1)k\pi)&= \cos\left(\frac{2j}{h+k} k \pi\right) = \cos\left(\frac{2j}{h+k} ((h+k)-h) \pi\right) \\
&=\cos\left(2j\pi - \frac{2j}{h+k}h\pi\right) = \cos\left(\frac{2j}{h+k}h\pi\right) \\
&= \cos((t_0-t_1)h\pi).
\end{align*}
The case $t_1=(2j+1)/(h+k)-t_0$ follows similarly.
\end{proof}

\begin{lem}\label{lem:small_u_torus_knot}
Assume that $h$ and $h+k$ are relatively prime. Then for $0 < u < \sqrt{k^2/(h^2+k^2)} $, $\gamma_u$ parameterizes an $(h,h+k)$-torus knot.
\end{lem}

\begin{proof}
The claim is clear by inspection upon plotting the explicitly parameterized $\gamma_u$. To prove it analytically, we consider the curve $\widetilde{\gamma}_u$ with parameterization

\begin{align*}
&\widetilde{\gamma}_u(t) = \frac{2}{\pi} \left( \frac{u\sqrt{1-u^2}}{h+k} \sin((h+k) \pi t),  \cos(h \pi t) \left(-\frac{u}{k}\cos((h+k) \pi t) + \frac{\sqrt{1-u^2}}{h}\right), \right. \\
&\hspace{3in} \left. \sin(h \pi t) \left(-\frac{u}{k} \cos((h+k) \pi t) + \frac{\sqrt{1-u^2}}{h}\right) \right).
\end{align*}
For $u < 1/2$, $\widetilde{\gamma}_u$ is an $(h,h+k)$-torus knot lying on a torus of revolution with eliptical cross-sections. Simplifying this formula using the trigonometric identities
$$
\mathrm{c}(h+k)\mathrm{c}(h) = \mathrm{c}(k) - \mathrm{s}(h)\mathrm{s}(h+k) \;\; \mbox{ and } \;\; \mathrm{c}(h+k) \mathrm{s}(h) = -\mathrm{s}(k) + \mathrm{c}(h)\mathrm{s}(h+k),
$$
we conclude that
$$
\widetilde{\gamma}_u = \gamma_u + \frac{u}{k} \left(0, \mathrm{s}(h)\mathrm{s}(h+k), \mathrm{c}(h) \mathrm{s}(h+k) \right).
$$
This decomposition can be used to give an isotopy from $\gamma_u$ to $\widetilde{\gamma}_u$ for small $u > 0$, and it follows that $\gamma_u$ is an $(h,h+k)$-torus knot.
\end{proof}

The same method can be used to show:

\begin{lem} 
Assume that $k$ and $h+k$ are relatively prime. Then for $ \sqrt{k^2/(h^2+k^2)} < u <  1$, $\gamma_u$ parameterizes a $(-k,h+k)$-torus knot.
\end{lem}

The results of this subsection are summarized as (cf. Theorem \ref{thm:ivey_singer}):

\begin{thm}\label{thm:1_param_families}
Let $h,k$ be integers such that $\mathrm{gcd}(h,h+k)=\mathrm{gcd}(k,h+k)=1$. The critical point set of $\mathrm{E}_\mathcal{M}$ with energy level $\pi (h^2 + k^2)$ contains a 1-parameter family of similarity classes of framed curves $(\gamma_u,V_u)$, $u \in [0,1]$, such that
\begin{itemize}
\item[(i)] $(\gamma_0,V_0)$ is an arclength parameterized $|h|$-times-covered round circle linked $k$-times,
\item[(ii)] $\gamma_u$ is a $(h,h+k)$-torus knot for $0 < u < \sqrt{k^2/(h^2+k^2)}$,
\item[(iii)] $\gamma_u$ is nonembedded for $u=\sqrt{k^2/(h^2+k^2)}$,
\item[(iv)] $\gamma_u$ is a $(-k,h+k)$ torus knot for $ \sqrt{k^2/(h^2+k^2)} < u <  1$,
\item[(v)] $(\gamma_1,V_1)$ is an arclength parameterized $|k|$-times-covered round circle linked $-h$-times.
\end{itemize}
\end{thm}

A $1$-parameter family of generalized rods is shown in Figure 1.

\begin{figure}\label{fig:1_param_family}
\centering
\begin{overpic}[width=0.85\textwidth,tics=5]{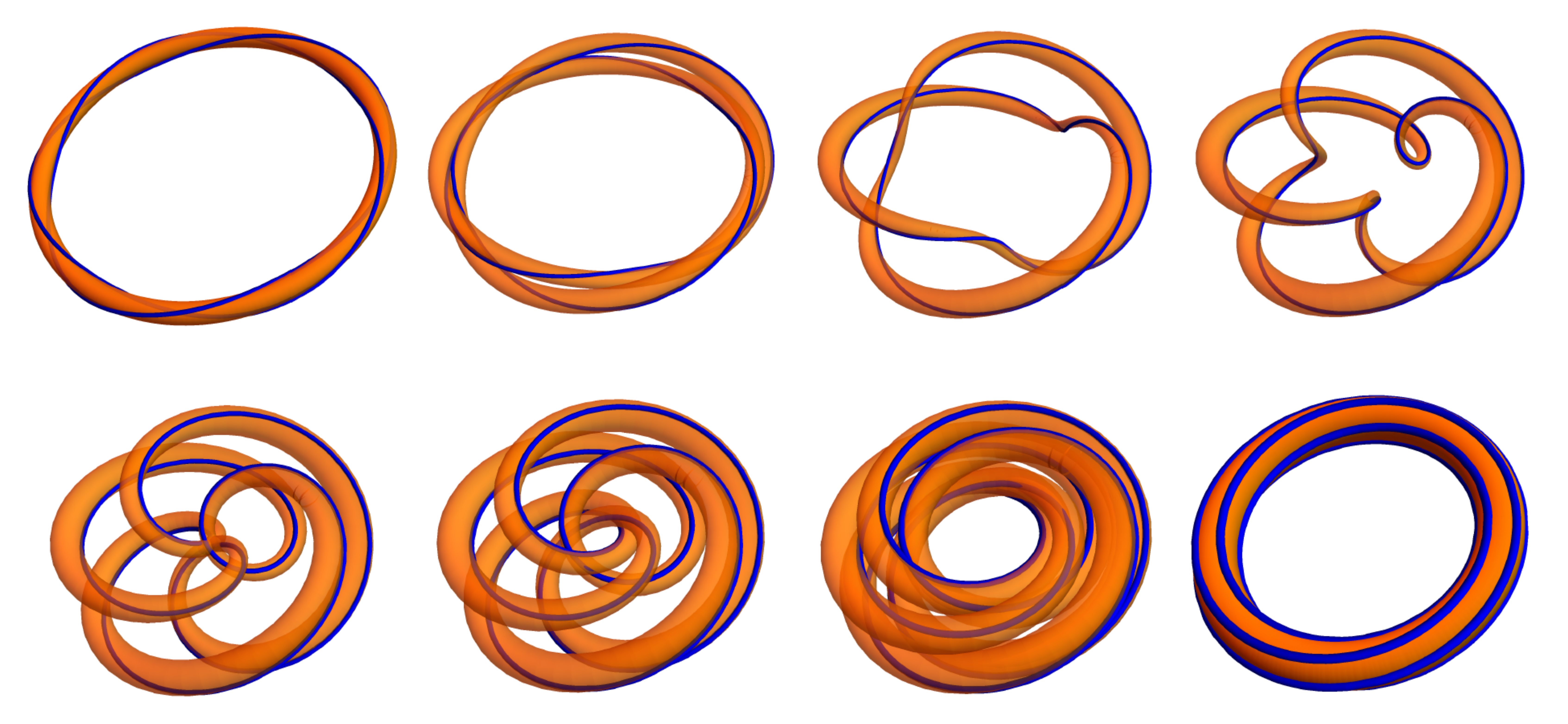}
 \put (9,22) {$u = 0$}
 \put (7,-2) {$u = 0.928$}
  \put (32,22) {$u = 0.15$}
 \put (31,-2) {$u = 0.964$}
   \put (58,22) {$u = 0.5$}
 \put (57,-2) {$u = 0.99$}
    \put (83,22) {$u = 0.8$}
 \put (84,-2) {$u = 1$}
\end{overpic}
\smallskip
\caption{The one-parameter family for parameters $(h,k)=(2,-5)$. Each framed curve is represented as a thickened tube with a line on its surface showing the twisting of the framing. The rods are displayed with variable thickness given by parameterization speed (see Section \ref{sec:generalized_energy_functional}). The pictured framed curves are spaced unevenly along the homotopy to better illustrate features of the evolution.}
\end{figure}
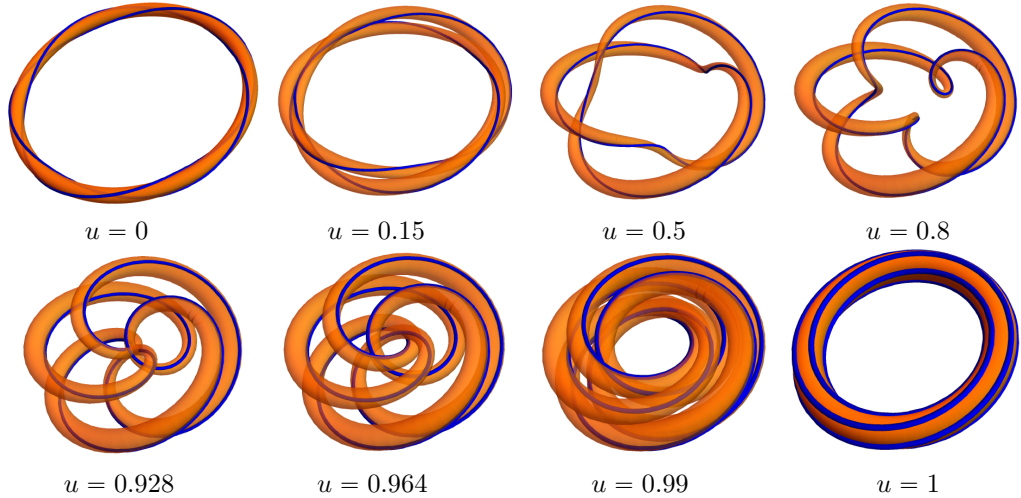

%%%%%%%%%%%%%%%%%%%%%%%%%%%%%%%%%%%%%%%%%%%%%%%%%%%%%%%%%%%%%%%%%%%
\subsection{Other Knot Types}\label{sec:other_knot_types}
%%%%%%%%%%%%%%%%%%%%%%%%%%%%%%%%%%%%%%%%%%%%%%%%%%%%%%%%%%%%%%%%%%%

The centerlines of critical points of $\mathrm{E}_\mathcal{M}$ in the $1$-parameter families described in Theorem \ref{thm:1_param_families} exhibit the same topologies as the classical Kirchhoff rods in Theorem \ref{thm:ivey_singer} of Ivey and Singer. On the other hand, the (non-isolated) critical sets of $\mathrm{E}_\mathcal{M}$ are five-dimensional, and one might expect richer topological variation within these sets than in the classical setting. 

Experimenting with parameters, one does find critical points of $\mathrm{E}_\mathcal{M}$ which are not torus knots. For example, the quaternionic curve 
$$
q\left(-3, 5, 0.09, i\sqrt{1-0.09^2}, 0.15,i\sqrt{1-0.15^2}\right)
$$
corresponds, under the frame-Hopf map, to a critical framed curve whose centerline forms the knot $10_{139}^\ast$. The centerline is parameterized as a trigonometric polynomial of degree $5$. Every knot type admits a parameterization as a trigonometric polynomial, and the lowest possible degree of such a parameterization is called the \emph{harmonic index} of the knot type \cite{kauffman, trautwein,trautwein2}. Trautwein shows in \cite{trautwein} that the harmonic index of a knot upper bounds its superbridge index, so this example gives an upper bound of $5$ for the superbridge index of $10_{139}^\ast$. Some of the non-torus knots that we discovered are illustrated in Figure 2.  Most of the more exotic knots that we discovered in the critical sets appear to be of high crossing number, making them more difficult to classify. All of this naturally leads to the following

\begin{question}
Besides torus knots, which knot types are realized as centerlines of critical points of $\mathrm{E}_\mathcal{M}$?
 \end{question}

\begin{figure}\label{fig:other_knots}
\centering
\begin{overpic}[width=0.85\textwidth,tics=5]{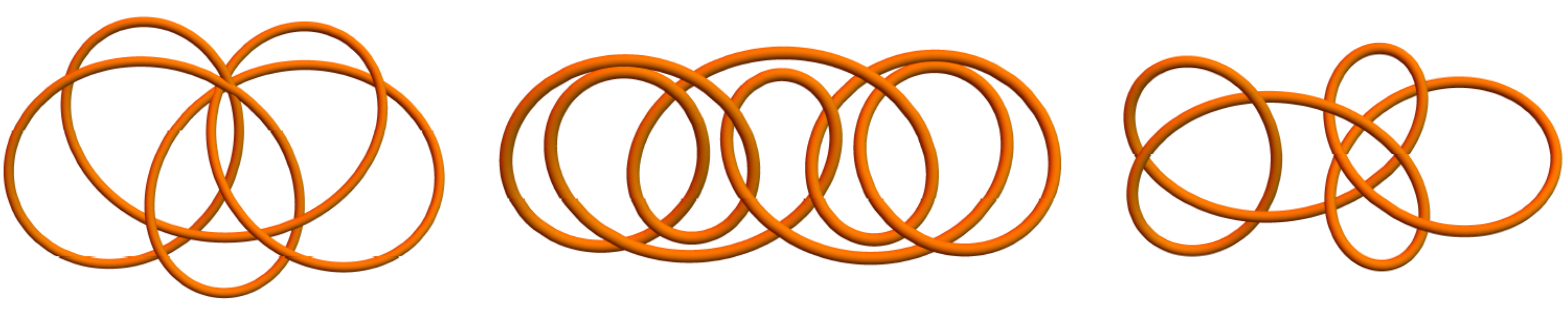}
\end{overpic}
\smallskip
\caption{Examples of non-torus knots which arise as centerlines of critical points of $\mathrm{E}_\mathcal{M}$. The knots displayed are $10_{139}^\ast$, $10_{152}$ and $3_1 \# 3_1$, respectively. Their respective quaternionic parameterizations have approximate parameter values $q(-3, 5, 0.09, 0.996i, 0.15, 0.989i)$, $q(5, 7, 0.986,0.167i,0.1-0.11i,-0.855+0.497i)$ and $q(-3,5,0.16,0.999i,-0.23,0.999i)$.}
\end{figure}

\section*{Acknowledgements}

I would like to thank Jason Cantarella for many illuminating conversations about spaces of curves and elastic energy and Aaron Trautwein for sharing a copy of his thesis with me.

\end{document}